\documentclass{article}

\usepackage[utf8]{inputenc}
\usepackage{color}
\usepackage{amsthm}
\usepackage{siunitx}
\usepackage{amsfonts}
\usepackage{amsmath}
\usepackage{latexsym}
\usepackage{tikz}
\usepackage{authblk}

\usetikzlibrary{decorations.pathreplacing}

\usepackage{epsfig}
\usepackage{graphicx, url}
\usepackage{cite}
\usepackage{color}
\usepackage{setspace}
\usepackage{caption}
\usepackage{subcaption}

\newtheorem{theorem}{Theorem}

\newtheorem{example}{Example}

\newtheorem{lemma}{Lemma}
\newtheorem{assumption}{Assumption}
\newtheorem{defin}{Definition}

\newtheorem{notation}{Notation}

\newcommand{\R}{\ensuremath{\mathbb{R}}}

\providecommand{\keywords}[1]
{
  \small	
  \textbf{\textit{Keywords---}} #1
}

\title{Stochastic control with self-exciting processes}
\author[1]{Heidar Eyjolfsson}
\author[2]{Kristina Rognlien Dahl}
\affil[1]{Reykjavik University}
\affil[2]{BI Norwegian Business School}
\begin{document}

\maketitle

\begin{abstract}
We analyze the problem of stochastic optimal control of SDEs where the driver includes a self-exciting stochastic process. Due to the non-Markovian nature of the problem, we apply the stochastic maximum principle approach. We derive a sufficient stochastic maximum principle under this framework. We also derive an expression via martingales of both the self-exciting process and its quadratic covariation. Furthermore, we derive a necessary maximum (equivalence principle) for the self-exciting stochastic control problem. Finally, we look at an application to log-utility.
\end{abstract}





\section{Introduction}
In this paper, we study stochastic optimal control of stochastic differential equations where the noise includes a self-exciting stochastic process. Since this self-exciting process is non-Markovian, we study the stochastic control problem via the maximum principle approach.

Hawkes \cite{Hawkes1} was the first to study self-exciting stochastic processes. The initial applications were in seismology. In recent time, variations of self-exciting processes have been used to model group behaviour in social media (see Rizoiu et al. \cite{R}), for predicting crime (see Mohler \cite{Mohler}) and for financial applications (see e.g. Bacry et al. \cite{Bacry} and Embrechts et al. \cite{Embrechts}). 

In this paper, our definition of self-exciting processes, follows Eyjolfsson and Tj\o stheim \cite{EyjolfssonTjostheim} and later Dahl and Eyjolfsson \cite{DahlEyjolfsson} as well as Ahmed and Eyjolfsson \cite{ahmed2026pricing}. The self-exciting process counts the number of occurred shocks (sometimes called events) at any given time. Associated to the self-exciting process is an intensity process, denoted by $\lambda = \{\lambda_t\}_{t \geq 0}$. This intensity process determines the probability of events occurring in the infinitesimal interval $(t,t+dt)$ conditioned on the information at time $t$. We remark that the self-exciting processes considered in this paper are different from Hawkes processes, see Hawkes \cite{Hawkes1} and Hawkes and Oakes \cite{HawkesOakes}. The difference is that the intensity verifies an SDE and the stochastic jump size of the self-exciting process may depend on the current value of the intensity process.

Stochastic optimal control has been widely studied over the past few decades, see e.g. Peng \cite{Peng} for the fundamental Itô diffusion case, or Framstad, Øksendal and Sulem \cite{Framstad} for the generalized Itô jump diffusion case. Based on these fundamental works, a large number of papers on stochastic control have emerged, where authors study different kinds of SDEs, SPDEs, jumps vs. no jumps, delay, memory and so on. Some recent works include Agram and {\O}ksendal \cite{AgramOksendal}, Makhlouf et al. \cite{Makhlouf} and Majee \cite{majee2023stochastic}. However, to the best of our knowledge, the problem of this paper, namely stochastic optimal control for SDEs with self-exciting driver, has not been analyzed before.

The paper is structured as follows: In Section \ref{sec: framework} we introduce the framework for self-exciting stochastic processes. This is the same framework as in Eyjolfsson and Tj\o stheim \cite{EyjolfssonTjostheim}, see also Eyjolfsson and Tj\o stheim \cite{ET23} for a multivariate version, and later Dahl and Eyjolfsson \cite{DahlEyjolfsson}. We include it here for the sake of completeness.  In Section \ref{sec: problem}, we introduce the stochastic optimal control problem with a self-exciting process as the driver of the SDE. We define a Hamiltonan based on this SDE, as well as the corresponding adjoint BSDE. In Section \ref{sec: lemmas}, we state and prove some lemmas that are needed in order to prove maximum principles for the self-exciting stochastic control problem. In particular, we prove that the self-exciting process and its quadratic covariation can be expressed via martingales. In Section \ref{sec: sufficient}, we prove a sufficient maximum principle for the self-exciting stochastic control problem. Then, in Section \ref{sec: necessary}, we prove a necessary maximum principle, also called an equivalence principle, for the self-exciting control problem. Finally, in Section \ref{sec:Application} we explore a log-utility application.

\section{The framework: Self-exciting stochastic processes}
\label{sec: framework}

The following definition of self-exciting processes is the same as in Eyjolfsson and Tj\o stheim \cite{EyjolfssonTjostheim}, see also Eyjolfsson and Tj\o stheim \cite{ET23} for a multivariate version, and later Dahl and Eyjolfsson \cite{DahlEyjolfsson}. It is included here for the sake of completeness. 

The self-exciting process is a counting process, counting the number of shocks that have occurred at a given time point. Let $(\Omega,\mathcal{F})$ be a measurable space and $\{T_n\}_{n \geq 1}$ be a point process with values in $\R_+$. We assume that the sequence $\{T_n\}_{n \geq 1}$ is non-negative and non-decreasing. Hence, $0 \leq T_1 \leq T_2 \leq \cdots$. This sequence represents times of successive shocks or events. Associated to the point process, there is a counting process $N_t$ which is defined by
\begin{equation}\label{def:N}
N_t := \sum_{n \geq 1} 1_{\{T_n \leq t\}}, \quad \mbox{ where } t \geq 0.
\end{equation}
This counting process counts the jumps of the point process. The rate at which jumps occur is governed by an intensity process, to be defined. We identify a point process with its counting process, see equation \eqref{def:N}, and define
\begin{equation*}
\mathcal{F}_t^N := \sigma\{N_s : 0 \leq s \leq t \}, \quad \mbox{ where } t \geq 0.
\end{equation*}
Hence, $\{\mathcal{F}_t^N\}_{t \geq 0}$ is the filtration generated by the counting process. 

Consider a point process adapted to some filtration $\{\mathcal{F}_t\}_{t \geq 0}$, with $\mathcal{F}_t^N \subset \mathcal{F}_t$ for all $t \geq 0$. Suppose that the corresponding counting process $N_t$ admits a c\`adl\`ag  $\{\mathcal{F}_t\}_{t \geq 0}$-adapted intensity process $\lambda_t$, such that  
\begin{equation}
\label{eq: intensity}
E\left[\int_0^\infty  f(s) dN_s \right] = E\left[\int_0^\infty f(s) \lambda_sds \right],
\end{equation}
holds for all predictable $f: \Omega \times \R_+ \to [-\infty,\infty]$, where $E$ denotes the expected value with respect to a given probability measure on $(\Omega,\mathcal{F})$. In particular, this means the the process $t \mapsto N_t - \lambda_t$ is a martingale. Furthermore, equation \eqref{eq: intensity} implies that the intensity process, $\lambda_t$, determines the probability of shocks occurring in the infinitesimal interval $(t,t+dt)$ conditioned on $\mathcal{F}_t$. If the intensity is constant, i.e., $\lambda_t = \lambda_0 > 0$, holds  for all $t \geq 0$, then $N_t$ is a standard homogeneous Poisson process with intensity $\lambda_0$. 

Let $\{Y_i\}_{i \in \mathbb{N}}$ be a family of random variables. We assume that $Y_i$ has the probability distribution $\nu(\lambda_{T_i-},\cdot)$, for a given family $\{\nu(\lambda,\cdot)\}_{\lambda > 0}$ of probability distributions, and $t- := \lim_{s \uparrow t} s$. That is, we let the value of the intensity process right before the jump affect the jump size distribution. Then, we define the stochastic jump process, 
\begin{equation}\label{def:U}
U_t := \sum_{i=1}^{N_t} Y_{i},
\end{equation}
where $\{N_t\}_{t \geq 0}$ is the counting process defined in equation \eqref{def:N}. We introduce the stochastic differential equation (SDE)

\begin{equation}\label{def:sde}
d\lambda_t = \mu(\lambda_t)dt +  \beta dU_t,
\end{equation}

\noindent $\lambda(0) = \lambda_0$, where $\beta \in \R$ is a constant and we assume that $\mu:\R_+ \to \R$, is Lipschitz continuous. 

\begin{defin}(SDE-driven self-exciting jump process)\label{def: self-exciting}
An SDE-driven self-exciting jump process, $U=\{U_t\}_{t \geq 0}$, is a stochastic jump process \eqref{def:U}, where the associated counting process, $N=\{N_t\}_{t \geq 0}$ admits the intensity $\lambda=\{\lambda_t\}_{t \geq 0}$, given by the SDE \eqref{def:sde}, with jump-sizes $\{Y_i\}_{i\geq1}$, which follow the probability distribution $\nu(\lambda_{T_i-},\cdot)$. Here, $\{\nu(\lambda,\cdot)\}_{\lambda > 0}$ is a family of probability distributions, and $\nu(\lambda,\cdot)$ is supported on $[\lambda_0-\lambda,\infty)$.
\end{defin}

In the next section, we define a stochastic control problem where a self-exciting jump process is the driver of the SDE.

\section{The self-exciting stochastic optimal control problem}
\label{sec: problem}

In the following let $U = U_t(\omega); (t, \omega) \in [0, \infty) \times \Omega$ be a self-exciting stochastic process, as defined in Definition \ref{def: self-exciting}. Also, let $B = B_t(\omega); (t, \omega) \in [0, \infty) \times \Omega$ be a Brownian motion on a complete filtered probability space $(\Omega, \mathcal{F}, \{\mathcal{F}_t\}_{t \geq 0}, P)$. We assume that $\mathbb{F} := \{\mathcal{F}_t\}_{t \geq 0}$ is the filtration generated by $B$, $N$ and $U$ (augmented with the $P$-null sets). 

Assume that $\mathcal{V} \subset \mathbb{R}$ and let $b, \sigma, \gamma$ be functions defined as follows:
\begin{align}b : [0,T] \times \mathbb{R} \times \mathcal{V} \times \Omega \rightarrow \mathbb{R},\\ 
\sigma : [0,T] \times \mathbb{R}  \times \mathcal{V} \times \Omega \rightarrow \mathbb{R},\\
\gamma : [0,T] \times \mathbb{R} \times \mathcal{V} \times \Omega \rightarrow \mathbb{R}.
\end{align} 
Note that we usually do not write the scenario $\omega$ (for notational convenience).

Consider the following stochastic differential equation with a self-exciting process as a driving noise:

\begin{equation}
\label{eq: SDE_SE}
    \begin{array}{llll}
        dX_t &=& b(t,X_t,\pi_t)dt + \sigma(t,X_t, \pi_t)dB_t + \gamma(t-,X_{t-},\pi_{t-})dU_t \\
         X_0 &=& x_0.  
    \end{array}
\end{equation}

This is the state equation. Here, the process $\pi = \{\pi_t(\omega)\}_{t \geq 0}$ is the control of the agent. We denote the set of admissible controls  by $\mathcal A$. We assume that this is some specified set of c\`{a}dl\`{a}g processes in $L^2(\Omega\times[0,T])$, with values in a subset $\mathcal V$ of $\mathbb R$. This implies that the usual conditions hold (in the sense of Protter \cite{Protter}).

We impose the following set of assumptions on the coefficient functions: 
\begin{assumption}
\label{assumption: existence_uniqueness}

\begin{enumerate}
\item\label{hyp:C1} The functions $(x,\pi) \mapsto b(t,x, \pi, \omega)$, $(x,\pi) \mapsto \sigma(t, x, \pi, \omega)$ and $(x,\pi) \mapsto \gamma(t, x, \pi, \omega)$ 
are assumed to be $C^1$ for fixed $t, \omega$ with locally bounded derivatives.

\item \label{hyp:measurable} The functions  $t \mapsto b(t, x,\pi)$, $t \mapsto \sigma(t, x,\pi)$ and $t \mapsto \gamma(t, x, \pi)$ are predictable, for each $(x,\pi) \in \mathbb{R} \times \mathcal{V}$.

\item \emph{Lipschitz condition:} The functions $b$, $\sigma$ and $\gamma$ 
are uniformly Lipschitz continuous in the variable $x$ for each $\pi \in \mathcal{V}$, with the Lipschitz constant  independent of the variables $t, \pi, \omega$. 

\end{enumerate}
\end{assumption}
Assumption  \ref{assumption: existence_uniqueness}.$\ref{hyp:C1}$ and Assumption \ref{assumption: existence_uniqueness}.$\ref{hyp:measurable}$ guarantee that the integrands in our state equation (\ref{eq: SDE_SE})
are predictable whenever $X$ is c\`{a}dl\`{a}g and adapted. Note also that linear growth is implied from the uniform Lipschitz-condition.

\medskip

\begin{theorem}
\label{thm: existence_unique}(Existence and uniqueness of solution to the state equation \eqref{eq: SDE_SE})
For each $\pi \in \mathcal{A}$, 
there exists a unique c\`{a}dl\`{a}g adapted global solution $X=X^{\pi}\in L^2(\Omega\times[0,T])$ to the state equation (\ref{eq: SDE_SE})  
\end{theorem}

\begin{proof}
Given Assumption \ref{assumption: existence_uniqueness}.$\ref{hyp:C1}$ and Assumption \ref{assumption: existence_uniqueness}.$\ref{hyp:measurable}$ combined with the the Lipschitz (and hence linear growth) conditions of Assumption \ref{assumption: existence_uniqueness}, the existence of a unique solution to the state equation \eqref{eq: SDE_SE} follows from semi-martingale stochastic integration theory, see Protter \cite{Protter} or Cohen and Elliott \cite{CohenElliott}.
\end{proof}
\medskip

We would like to maximize the following performance functional over all controls $\pi \in \mathcal{V}$, given the SDE \eqref{eq: SDE_SE} for the state process $X$,

\begin{equation}
\label{eq: OptControl}
    \max_{\pi \in \mathcal{A}} \quad J(\pi) = E[\int_0^T h(t,X_t, \pi_t)dt + g(X_T)].
\end{equation}

\noindent Here, $E[\cdot]$ denotes expectation with respect to $P$ and $$h :[0,T] \times \mathbb{R} \times \mathcal{V} \times \Omega \rightarrow \mathbb{R}$$ and $$g : \mathbb{R} \times \Omega \rightarrow \mathbb{R}$$ are given $C^1$ functions, satisfying the following assumptions:

\begin{assumption}
\label{assumption: performance}
The functions $h$ and $g$ satisfy the following:
 \begin{enumerate}
\item  The functions $ (x,\pi) \mapsto h(t,x,\pi, \omega)$ and $x \mapsto g(x, \omega)$ are $C^1$ for each $t,\omega$.
 \item The function $ t \mapsto h(t, x, \pi)$ is predictable and $ g(x)$ is $\mathcal F_T$-measurable.
\end{enumerate}
\end{assumption}

Because the self-exciting process, $U$, is non-Markovian, we use the stochastic maximum principle approach for the optimization. To do so, we define the Hamiltonian function

\begin{equation}
\label{eqref: Hamiltonian}
\begin{array}{lllll}
    \mathcal{H}(t,x,\pi,p,q,w) &=& h(t,x,\pi) + \big( b(t,x,\pi) + \sum_{i=1}^{N_t} Y_i \boldsymbol{1}_{(T_{i-1},T_i]}(t) \lambda_t \gamma(t,x,\pi) \big) p_t \\[\medskipamount]
     && + \sigma(t,x,\pi)q_t + \lambda_t \big(\sum_{i=1}^{N_t} Y_i^2 \boldsymbol{1}_{(T_{i-1},T_i]}(t) \gamma(t,x,\pi) \\[\medskipamount]
     &&+\sum_{i=1}^{N_t} Y_i \boldsymbol{1}_{(T_{i-1},T_i]}(t)x \big)w_t 
\end{array}
\end{equation}

\noindent To ease notation, we introduce the following shorthand:

\begin{notation}
\label{not: shorthand}
We define the following shorthand notation for some given control $\pi$ with corresponding solution $X$ to the state equation \eqref{eq: SDE_SE}:\newline
$$b_t := b(t, X_t, \pi_t), \sigma_t := \sigma(t, X_t, \pi_t),\gamma_{t} := \gamma(t,X_{t^{-}}, \pi_{t^{-}}),h_t := h(t, X_t, \pi_t),$$ $$\frac{\partial b_t}{\partial x} := \frac{\partial b}{\partial x}(t, X_t, \pi_t),\frac{\partial \sigma_t}{\partial x} := \frac{\partial \sigma}{\partial x}(t, X_t, \pi_t),\frac{\partial \gamma_{t}}{\partial x} :=\frac{\partial \gamma}{\partial x} (t,X_{t^{-}}, \pi_{t-}),$$
$$\frac{\partial h_t}{\partial x} := \frac{\partial h}{\partial x}(t, X_t, \pi_t), \frac{\partial \mathcal{H}_t}{\partial x}:=\frac{\partial \mathcal{H}}{\partial x}(t, X_{t}, \pi_{t}, {p}_t, {q}_t, {w}_t),$$
\end{notation}

In equation \eqref{eqref: Hamiltonian}, $p, q, w$ are the adjoint processes satisfying the following adjoint backward stochastic differential equation (BSDE),

\begin{equation}
\label{eq: BSDE}
    \begin{array}{llll}
        dp_t &=& -\frac{\partial \mathcal{H}_t}{\partial x} dt + q_t dB_t + w_{t-} d U_t \\[\medskipamount]
         p_T &=& g'(X_T). 
    \end{array}
\end{equation}

Note that the BSDE \eqref{eq: BSDE} has the self-exciting process $U_t$ as its driving noise. In general, $U_t$ is not a martingale, only a semi-martingale. This means that, in its current form, standard results on existence and uniqueness of a solution to this equation do not apply. However, this is an active field of research and there are some recent papers on BSDEs with semi-martingale noise. Jeanblanc et al. \cite{jeanblanc2012mean} consider a stochastic control problem which leads to a particular kind of BSDEs with non-standard noise: Elmansouri and El Otmani \cite{elmansouri2023generalized} consider BSDEs where the noise is given by an c\`adl\`ag martingale. Dumitrescu et al. \cite{dumitrescu2016bsdes} consider BSDEs with one single default jump and Papapantoleon et al. \cite{papapantoleon2018existence} give existence and uniqueness results for BSDEs driven by a general martingale. 

In the forthcoming Lemma \ref{lemma: martingales}, we will show that it is possible to express $U_t$ via a martingale. This characterization may be used to rewrite the semi-martingale BSDE (\ref{eq: BSDE}) into a BSDE involving a drift term, a Brownian noise term and a martingale term. Existence and uniqueness results for such equations have been studied by Papapantoleon et al. \cite{papapantoleon2018existence} and Elmansouri and El Otmani \cite{elmansouri2023generalized}. Note that the BSDE \eqref{eq: BSDE} has a linear structure because

\begin{equation}
\label{eq:dHdx}
\begin{array}{llll}
\frac{\partial \mathcal{H}}{\partial x} &=& \frac{\partial h}{\partial x} +\big(\frac{\partial b}{\partial x} + \sum_{i=1}^{N_t}Y_i\boldsymbol{1}_{(T_{i-1},T_i]}(t)\lambda\frac{\partial \gamma}{\partial x}\big)p_t \\[\medskipamount] 
&&+ \frac{\partial \sigma}{\partial x} q_t + \lambda_{t} w_{t} \sum_{i=1}^{N_t} (Y_i + Y_i^2\frac{\partial \gamma}{\partial x}) \boldsymbol{1}_{(T_{i-1},T_i]}(t) 
\end{array}
\end{equation}
\noindent which is linear in $p_t$, $q_t$ and $w_t$.

\section{A sufficient stochastic maximum principle}
\label{sec: sufficient}
In this section, we prove a sufficient maximum principle for self-exciting stochastic optimal control. 

Under this framework, we can prove the following sufficient maximum principle. In order to prove this result, we will need some lemmas. We state the sufficient maximum theorem here, the we state and prove the required lemmas in Section \ref{sec: lemmas} before proving the sufficient maxmimum principle in Section \ref{sec: sufficient_proof}.

\begin{theorem} (The sufficient maximum principle)
\label{thm: suff_max_princ}
Let $\hat{\pi}$ be an admissible performance strategy with corresponding solution $\hat{X}$ of the self-exciting state SDE \eqref{eq: SDE_SE} and assume that $(\hat{p}, \hat{q}, \hat{w})$ solves the corresponding adjoint BSDE \eqref{eq: BSDE}. Furthermore, assume that
\begin{enumerate}
    \item $g$ and $(x, \pi) \mapsto \mathcal{H}(t,x,\pi,\hat{p},\hat{q}, \hat{w})$ are concave a.s. and
    \item $\max_{\pi \in \mathcal{V}} \mathcal{H}(t,\hat{X},\pi,\hat{p},\hat{q},\hat{w}) = \mathcal{H}(t, \hat{X}, \hat{\pi}, \hat{p}, \hat{q}, \hat{w})$ \quad $dt \times P$ \mbox{ a. s.}
\end{enumerate}
Then, $\hat{\pi}$ is an optimal control for the self-ecxciting stochastic optimal control problem \eqref{eq: OptControl}.
\end{theorem}
\medskip

Before we go on to state and show some lemmas needed in the proof of Theorem \ref{thm: suff_max_princ}, note that the theorem says that under its concavity assumptions, we can maximize the Hamiltonian instead of the performance functional in order to solve our self-exciting stochastic control problem.

\subsection{Some lemmas}
\label{sec: lemmas}
In this section, we prove some lemmas that are needed in order to prove Theorem \ref{thm: suff_max_princ}.

In the proof of the sufficient maximum principle, Theorem \ref{thm: suff_max_princ}, we need to compute the quadratic covariation of two processes: The adjoint process, $\hat{p}$ and the differences of the state processes, $X-\hat{X}$. This is our first lemma: 

\begin{lemma}(The quadratic covariation)
\label{lemma: quadratic_variation}
For the adjoint process, $\hat{p}$, and the difference in state process, $X-\hat{X}$, the quadratic covariation process is given as follows:

\[
d[\hat{p}, (X-\hat{X})]_t =  \hat{q}_t (\sigma_t - \hat{\sigma}_t)dt + \hat{w}_{t-}(\gamma_{t-} - \hat{\gamma}_{t-})d[U]_t.
\]

\end{lemma}

\begin{proof}
From the polarization identity,

\begin{equation}
    \label{eq: polarization}
[\hat{p}, (X-\hat{X})]_t = \frac{1}{2}([\hat{p}+(X - \hat{X})]_t - [\hat{p}]_t - [X - \hat{X}]_t).
\end{equation}

To prove the lemma, we compute each term in equation \eqref{eq: polarization} separately: Note that, from the definition of the quadratic covariation

\[
[\hat{p}]_t = \lim_{|P| \rightarrow 0} \sum_{k=1}^n (\Delta \hat{p}_{t_k})^2 .
\]

From the definition of the adjoint BSDE, equation \eqref{eq: BSDE}, we have that the terms $(\Delta \hat{p}_{t_k})^2$ in the sum above consist of six types of sub-terms: $(\Delta t_k)^2$, $\Delta t_k \Delta B_{t_k}$, $\Delta t_k \Delta U_{t_k}$, $\Delta B_{t_k} \Delta U_{t_k}$, $(\Delta U_{t_k})^2$ and finally $(\Delta B_{t_k})^2$, all of which are multiplied with their respective coefficient functions given by the BSDE \eqref{eq: BSDE}. Apart from $(\Delta U_{t_k})^2$ and $(\Delta B_{t_k})^2$, all of these sub-terms disappear: Since $\Delta t_k$ goes to $0$ as the mesh grows finer and $U$ is a finite variation process, the cross terms $\Delta t_k \Delta U_{t_k}$ will disappear. Furthermore, the same holds for the cross terms $\Delta B_{t_k} \Delta U_{t_k}$ since $B$ is a martingale and $U$ has finite variation. Hence,

\begin{equation}
\label{eq: quad_var_0}
d[\hat{p}]_t = \hat{q}_t^2 dt + \hat{w}_{t-}^2 d[U]_t .
\end{equation}

By parallel arguments as above and the self-exciting state SDE \eqref{eq: SDE_SE}, 

\begin{equation}
\label{eq: quad_var1}
d[X-\hat{X}]_t = (\sigma_t - \hat{\sigma}_t)^2 dt + (\gamma_{t-} - \hat{\gamma}_{t-})^2 d[U]_t.
\end{equation}

Note also that, by the self-exciting state equation \eqref{eq: SDE_SE} and the adjoint BSDE \eqref{eq: BSDE} in combination, we find that:

\begin{equation}
\label{eq: quad_var2}
\begin{array}{llll}
d[\hat{p} + (X - \hat{X})]_t &=& (\hat{q}_t+(\sigma_t - \hat{\sigma}_t))^2 dt + (\hat{w}_{t-} + (\gamma_{t-} - \hat{\gamma}_{t-}))^2 d[U]_t \\[\medskipamount]
&=& \{ \hat{q}_t^2 + 2 \hat{q}_t (\sigma_t - \hat{\sigma}_t) + (\sigma_t - \hat{\sigma}_t)^2 \}dt \\[\medskipamount]
&&+ \{ \hat{w}_{t-}^2 + 2 \hat{w}_{t-} (\gamma_{t-} - \hat{\gamma}_{t-}) + (\gamma_{t-} - \hat{\gamma}_{t-})^2 \}d[U]_t.
\end{array}
\end{equation}

By inserting equations \eqref{eq: quad_var_0}, \eqref{eq: quad_var1} and \eqref{eq: quad_var2} into the expression in \eqref{eq: polarization}, we see that 

\begin{equation}
\begin{array}{llll}
d[\hat{p}, (X-\hat{X})]_t &=&  \frac{1}{2} \Big( \{ \hat{q}_t^2 + 2 \hat{q}_t (\sigma_t - \hat{\sigma}_t) + (\sigma_t - \hat{\sigma}_t)^2 \}dt \\[\medskipamount]
&&+ \{ \hat{w}_{t-}^2 + 2 \hat{w}_{t-} (\gamma_{t-} - \hat{\gamma}_{t-}) + (\gamma_{t-} - \hat{\gamma}_{t-})^2 \}d[U]_t - \hat{q}_t^2 dt  \\[\medskipamount]
&&- \hat{w}_{t-}^2 d[U]_t - (\sigma_t - \hat{\sigma}_t)^2 dt - (\gamma_{t-} - \hat{\gamma}_{t-})^2 d[U]_t\Big) \\[\medskipamount]

&=& \hat{q}_t (\sigma_t - \hat{\sigma}_t)dt + \hat{w}_t(\gamma_{t-} - \hat{\gamma}_{t-})d[U]_t.
\end{array}
\end{equation}

\end{proof}

In order to prove a representation of the self-exciting process via martingales, we need the following result by Brémaud \cite{Br81}. The result is included here for completeness.

\begin{theorem}(Brémaud \cite{Br81}, Theorem T8)
\label{thm: bremaud}
Assume that $N_t$ admits the intensity $\lambda_t$, and define $M_t:= N_t - \int_0^t \lambda_sds$. Then, if $X_t$ is $\mathcal{F}_t$-predictable such that $E[\int_0^t |X_s|\lambda_sds] < \infty$, then $\int_0^t X_sdM_s$ is an $\mathcal{F}_t$-martingale.
\end{theorem}

In proving the representation of the self-exciting process via martingales, we will define a process of the form $\bar{X}_s := \sum_{i=1}^{N_s} Y_i \boldsymbol{1}_{(T_{i-1},T_i]}(s)$. The following example shows a toy illustration of this process compared to the self-exciting process $U_t$.

\begin{example}(Toy illustration of the self-exciting process $U_t$ and the process $\bar{X}_t$)
\label{ex: toy_example}
Let 
\begin{equation}
\bar{X}_s := \sum_{i=1}^{N_s} Y_i \boldsymbol{1}_{(T_{i-1},T_i]}(s)    
\end{equation}
\noindent where $Y_i$ are the $\mathcal{F}_{T_{i-1}}$-measurable jump sizes and $0 \leq T_0 \leq \ldots \leq T_{N_s}$ are the jump times of $N_s$ (so $N_{T_j}- N_{T_{j-1}} = 1$ for all $j \geq 1$). 

In Figure \ref{fig: toy_example}, you see an illustration of one path of a (toy) self-exciting process $U_t$.

\begin{figure}[h!]
\resizebox{\textwidth}{!}{%
\begin{tikzpicture}[
    font=\small, 
    thick, 
]
    \coordinate (v1) at (0, 6);     
    \coordinate (v2) at (0, 5);     
    \coordinate (lvl_34) at (0, 3);  
    \coordinate (lvl_23) at (0, 2);  

    \coordinate (t0) at (1, 0);
    \coordinate (t1) at (2.5, 0); 
    \coordinate (t2) at (4, 0);   
    \coordinate (t3) at (8, 0);   
    \coordinate (t4) at (9.5, 0); 
    
    \coordinate (s) at (3.25, 0);  
    \coordinate (stilde) at (11.5, 0); 
    
    \draw[->] (-1, 0) -- (13, 0) node[below, xshift=3mm] {$t$};
    \draw[->] (0, -1) -- (0, 7) node[left] {}; 

    \foreach \tname/\n in {t0/0, t1/1, t2/2, t3/3, t4/4} {
        \draw (\tname) ++(0, 3pt) -- ++(0, -3pt) 
            node[below, yshift=-3pt] {$T_{\n}$};
    }
    
    \draw (s) ++(0, 3pt) -- ++(0, -3pt) node[below, yshift=-3pt] {$s$};
    \draw (stilde) ++(0, 3pt) -- ++(0, -3pt) node[below, yshift=-3pt] {$\tilde{s}$};
    
    \draw (3pt, 6) -- (-3pt, 6) node[left] {$v_1$};
    \draw (3pt, 5) -- (-3pt, 5) node[left] {$v_2$};

    \draw[dashed] (t0) -- (t0 |- v2);
    \draw[dashed] (t1) -- (t1 |- v1);
    \draw[dashed] (t2) -- (t2 |- lvl_23);
    \draw[dashed] (t3) -- (t3 |- lvl_34);
    
    \draw[dashed] (t4) -- (t4 |- v2);

    \draw[dashed] (v1) -- (t2 |- v1);
    \draw[dashed] (v2) -- (t1 |- v2);
    
    \draw[dashed] (s) -- (s |- v1);
    
    \draw[dashed] (stilde) -- (stilde |- v2);
    
    \draw[dashed] (stilde |- v2) -- (v2);

    \definecolor{funcblue}{HTML}{00529F}
    
    \draw[thick, funcblue] (t0 |- v2) -- (t1 |- v2);
    \draw[thick, funcblue] (t1 |- v1) -- (t2 |- v1);
    \draw[thick, funcblue] (t2 |- lvl_23) -- (t3 |- lvl_23);
    \draw[thick, funcblue] (t3 |- lvl_34) -- (t4 |- lvl_34);
    
    \draw[thick, funcblue] (t4 |- v2) -- (12.5, 5); 
    
    \fill[funcblue] (t0 |- v2) circle (1.5pt);
    \fill[funcblue] (t1 |- v1) circle (1.5pt);
    \fill[funcblue] (t2 |- lvl_23) circle (1.5pt);
    \fill[funcblue] (t3 |- lvl_34) circle (1.5pt);
    \fill[funcblue] (t4 |- v2) circle (1.5pt); 
    
    \draw[thick, funcblue, dotted] (t1 |- v2) -- (t1 |- v1);
    \draw[thick, funcblue, dotted] (t2 |- v1) -- (t2 |- lvl_23);
    \draw[thick, funcblue, dotted] (t3 |- lvl_23) -- (t3 |- lvl_34);
    \draw[thick, funcblue, dotted] (t4 |- lvl_34) -- (t4 |- v2);

    \draw [decorate, decoration={brace, amplitude=4pt, mirror}] 
        ([xshift=2mm]t0) -- ([xshift=2mm]t0 |- v2) 
        node[midway, left, xshift=-3mm] {$Y_1$};
        
    \draw [decorate, decoration={brace, amplitude=4pt, mirror}] 
        ([xshift=2mm]t1 |- v2) -- ([xshift=2mm]t1 |- v1) 
        node[midway, left, xshift=-3mm] {$Y_2$};

    \draw [decorate, decoration={brace, amplitude=4pt, mirror}] 
        ([xshift=2mm]t2 |- lvl_23) -- ([xshift=2mm]t2 |- v1) 
        node[midway, left, xshift=-3mm] {$Y_3$};

    \draw [decorate, decoration={brace, amplitude=4pt, mirror}] 
        ([xshift=2mm]t3 |- lvl_23) -- ([xshift=2mm]t3 |- lvl_34) 
        node[midway, left, xshift=-3mm] {$Y_4$};

    \draw [decorate, decoration={brace, amplitude=4pt, mirror}] 
        ([xshift=2mm]t4 |- lvl_34) -- ([xshift=2mm]t4 |- v2) 
        node[midway, left, xshift=-3mm] {$Y_5$};

    \node at (10.5, 5.7) {$U_t$};
\end{tikzpicture}
}
\caption{Illustration of a toy example of a self-exciting process $U_t$.}
\label{fig: toy_example}
\end{figure}
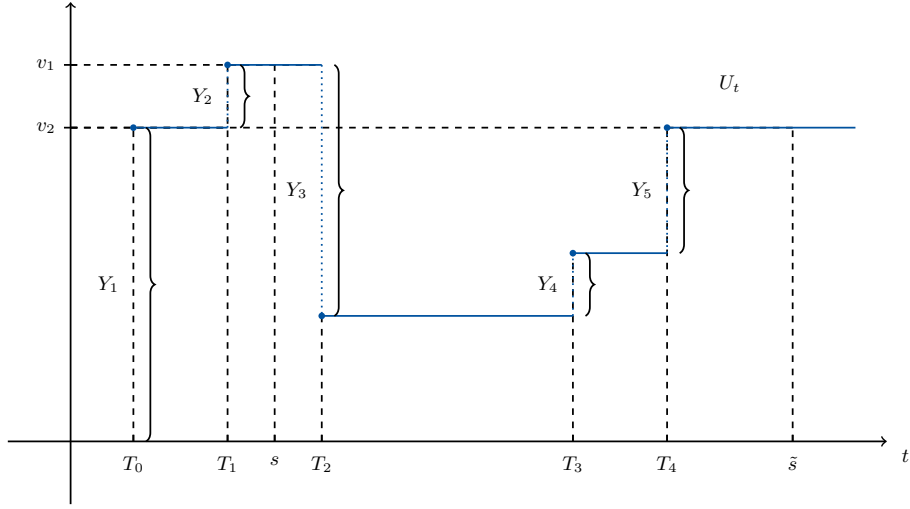

What is the corresponding process $\bar{X}_t$ in this case? Let's look at two points, $s$ and $\tilde{s}$ for illustration.

\begin{align*}
\bar{X}_s &= \sum_{i=1}^{N_s} Y_i \boldsymbol{1}_{(T_{i-1},T_i]}(s) = \sum_{i=1}^{2} Y_i \boldsymbol{1}_{(T_{i-1},T_i]}(s) \\
&=  Y_1 \boldsymbol{1}_{(T_{0},T_1]}(s) +  Y_2 \boldsymbol{1}_{(T_{1},T_2]}(s)=Y_2 = Y_{N_s}
\end{align*}

\noindent where the second equality follows because the number of jumps at time $s$ is $2$, and the final equality follows since $s$ is between jump times $T_1$ and $T_2$. The same kind of calculation can be done in the point $\tilde{s}$.

Clearly, $U_s=v_1$ and $U_{\tilde{s}}=v_2$. Hence, in particular, $\bar{X}_s \neq U_s$ (and $\bar{X}_{\tilde{s}} \neq U_{\tilde{s}}$), so the processes differ.

In fact, it holds in general, from the definition of $\bar{X}_s$, that 

\[
\bar{X}_s = Y_{N_s}.
\]

\end{example}

\bigskip

Now, we prove that the self-exciting process $U_t$ and its quadratic variation process $[U]_t$ can be expressed via martingales. This is crucial for the proof of the sufficient maximum principle, Theorem \ref{thm: suff_max_princ}. 

\begin{lemma}($U_t$ and $[U]_t$ can be expressed via martingales)
   \label{lemma: martingales}
Assume that \\$E\big[\int_0^t Y_i^2 \lambda_sds\big] < \infty$ for all $t<\infty$ and $Y_i \in \{Y_i\}_{i \in \mathbb{N}}$. Define
\[
\begin{array}{llll}
\bar{M}_t &:=& U_t -  \sum_{i=1}^{N_t} Y_i \int_{T_{i-1}}^{T_i} \lambda_s ds \\[\medskipamount]
\tilde{M}_t &:=& [U]_t - \sum_{i=1}^{N_t} Y_i^2 \int_{T_{i-1}}^{T_i} \lambda_s ds .
\end{array}
\]
Then, $\bar{M}_t$ and $\tilde{M}_t$ are $\mathcal{F}_t$-martingales.
\end{lemma}

\begin{proof}

The idea of the proof is to define a suitable process, $\bar{X}_s$, such that when we integrate this process w.r.t. the counting process $N_t$, we get the self-exciting process $U_t$. We also need to ensure that this process, $\bar{X}$, fits into the framework of Theorem \ref{thm: bremaud}. Then, by using the definition of the compensated martingale, $M_t:= N_t - \int_0^t \lambda_sds$, we can express $U_t$ via the martingale $\int_0^t \bar{X}_sdM_s$. In order to express the quadratic variation of the self-exciting process, $[U]_t$, via a martingale, we follow the same approach. However, the process we define, denoted by $\tilde{X}_s$, must integrate to the quadratic variation when we integrate w.r.t. the counting process.

\medskip

First we show that the self-exciting process $U_t$ can be expressed via a martingale. Let 
\[
\bar{X}_s := \sum_{i=1}^{N_s} Y_i \boldsymbol{1}_{(T_{i-1},T_i]}(s)
\]
\noindent where $Y_i$ are the $\mathcal{F}_{T_{i-1}}$-measurable jump sizes and $0 \leq T_0 \leq \ldots \leq T_{N_s}$ are the jump times of $N_s$ (so $N_{T_j}- N_{T_{j-1}} = 1$ for all $j \geq 1$). Note that

\[
\bar{X}_s = Y_{N_s},
\]
see Example \ref{ex: toy_example} and the comments therein. Then, $\bar{X}$ is $\mathcal{F}_t$-predictable because it is adapted and left-continuous and 

\begin{equation}
\label{eq: mellom}    
\begin{array}{lllll}
\int_0^t \bar{X}_s dN_s &=& \sum_{s \leq t} \bar{X}_{s-} \Delta N_s \\[\medskipamount]

&=& \sum_{i=1}^{N_t} \bar{X}_{T_i} \Delta N_{T_i} \\[\medskipamount]

&=& \sum_{i=1}^{N_t} Y_i \\[\medskipamount]

&=& U_t.

\end{array}
\end{equation}

\noindent Thus, by Theorem \ref{thm: bremaud},

\begin{equation}
    \begin{array}{llll}
        \int_0^t \bar{X}_{s-}dM_s &=& \int_0^t \bar{X}_{s-} \big( dN_s - \lambda_s ds \big) \\[\medskipamount]
        
        &=& \int_0^t \bar{X}_{s-} dN_s - \int_0^t \bar{X}_s \lambda_s ds \\[\medskipamount]
        
        &=& U_t - \int_0^t \big( \sum_{i=1}^{N_s} Y_i \boldsymbol{1}_{(T_{i-1}, T_i]}(s) \big) \lambda_s ds \\[\medskipamount]

        &=& U_t - \sum_{i=1}^{N_t} Y_i \int_{T_{i-1}}^{T_i} \lambda_s ds \\[\medskipamount]

        &=& \bar{M}_t.
    \end{array}
\end{equation}
\noindent where the third equality follows from equation \eqref{eq: mellom}. Because of the boundedness assumption, it follows from Theorem \ref{thm: bremaud},  that $\bar{M}_t$ is a martingale because $\int_0^t \bar{X}_sdM_s$ is a martingale.

\medskip

Then, we show that the quadratic covariation of the self-exciting process, $[U]_t$, can be expressed via a (different) martingale. To do so, define

\begin{equation}
    \label{eq: X_tilde}
    \tilde{X}_t := \sum_{i=1}^{N_t} Y_i^2 \boldsymbol{1}_{(T_{i-1}, T_i]}(t)
\end{equation}
\noindent where $Y_i$ is defined as before. Then, $\tilde{X}$ is $\mathcal{F}_t$-predictable, because it is adapted and left continuous. Also,

\begin{equation}
\label{eq: int_X_tilde}
    \begin{array}{llll}
        \int_0^T \tilde{X}_{s-} dN_s &=& \sum_{s \leq t} \tilde{X}_{s-} \Delta N_s \\[\medskipamount]

        &=& \sum_{i=1}^{N_t} \tilde{X}_{T_i} \Delta N_{T_i} \\[\medskipamount]

        &=& \sum_{i=1}^{N_t} Y_i^2 \\[\smallskipamount]

        &=& [U]_t
    \end{array}
\end{equation}
\noindent where the final equality holds because the jump process is constant between the jumps 

Hence,

\begin{equation}
    \begin{array}{llll}
        \int_0^t \tilde{X}_{s-}dM_s &=& \int_0^t \tilde{X}_{s-} \big( dN_s - \lambda_s ds \big) \\[\medskipamount]

        &=& \int_0^t \tilde{X}_{s-} dN_s - \int_0^t \tilde{X}_s \lambda_s ds \\[\medskipamount]

        &=& [U]_t - \int_0^t \big( \sum_{i=1}^{N_s} Y_i^2 \boldsymbol{1}_{(T_{i-1}, T_i]}(s) \big) \lambda_s ds \\[\medskipamount]

        &=& [U]_t - \sum_{i=1}^{N_t} Y_i^2 \int_{T_{i-1}}^{T_i} \lambda_s ds \\[\medskipamount]

        &=& \tilde{M}_t
    \end{array}
\end{equation}
\noindent where the third equality follows from equation \eqref{eq: int_X_tilde}. Note that because of the boundedness assumption in the lemma,  $\int_0^t \tilde{X}_sdM_s$ is a martingale by Theorem \ref{thm: bremaud}. Hence, $\tilde{M}$ is a martingale as well, from the calculation above.
\end{proof}

Note that Lemma \ref{lemma: martingales} also implies that the self-exicing BSDE \eqref{eq: BSDE} can be rewritten via martingales:

\[
    \begin{array}{llll}
        dp_t &=& -\frac{\partial \mathcal{H}}{\partial x} dt + q_t dB_t + w_{t-} d U_t \\[\medskipamount]
        &=& -\frac{\partial \mathcal{H}}{\partial x}  dt + q_t dB_t + w_{t-} d\big( \bar{M}_t + \sum_{i=1}^{N_t} Y_i \int_{T_{i-1}}^{T_i} \lambda_s ds \big) 
    \end{array}
\]
\noindent with the terminal condition $p_T=g'(X_T)$. Recall that the expression for $\frac{\partial \mathcal{H}}{\partial x}$ is given in equation \eqref{eq:dHdx}.  

We remark that after the rewriting via martingales, we still have a linear BSDE. As previously mentioned, existence and uniqueness results for such equations have been studied by Papapantoleon et al. \cite{papapantoleon2018existence} and Elmansouri and El Otmani \cite{elmansouri2023generalized}.

\subsection{Proof of the sufficient maximum principle}
\label{sec: sufficient_proof}

With Lemma \ref{lemma: quadratic_variation} and Lemma \ref{lemma: martingales} in place, we are ready to prove the sufficient maximum principle, Theorem \ref{thm: suff_max_princ}:

\medskip

\begin{proof}(\emph{Proof of Theorem \ref{thm: suff_max_princ}: The sufficient maximum principle})\\[\smallskipamount]
\noindent Fix an admissible control $\hat{\pi} \in \mathcal{A}$ with corresponding solutions $\hat{X}_t, \hat{p}_t, \hat{q}_t, \hat{w}_t$ to the self-exciting state SDE \eqref{eq: SDE_SE} and the corresponding adjoint BSDE \eqref{eq: BSDE}. We will use the following short-hand notation:
$$
\hat{b}_t := b(t, \hat{X}_t, \hat{\pi}_t), \mbox{ }
\hat{\sigma}_t := \sigma(t, \hat{X}_t, \hat{\pi}_t), \mbox{ }
\hat{\gamma}_{t} := \gamma(t, \hat{X}_{t^{-}}, \hat{\pi}_{t-}), \mbox{ }
\hat{h}_t := h(t, \hat{X}_t, \hat{\pi}_t),
$$
$$
\mathcal{H}_t:= \mathcal{H}(t, X_{t}, \pi_{t}, \hat{p}_t, \hat{q}_t, \hat{w}_t), \mbox{ } \hat{\mathcal{H}}_t :=\mathcal{H}(t, \hat{X}_{t}, \hat{\pi}_{t}, \hat{p}_t, \hat{q}_t, \hat{w}_t),
$$
$$
\frac{\partial \hat{b}_t}{\partial x} := \frac{\partial b}{\partial x}(t, \hat{X}_t, \hat{\pi}_t), \mbox{ }
\frac{\partial \hat{\sigma}_t}{\partial x} := \frac{\partial \sigma}{\partial x}(t, \hat{X}_t, \hat{\pi}_t), \mbox{ }
\frac{\partial \hat{\gamma}_{t}}{\partial x} :=\frac{\partial \gamma}{\partial x} (t, \hat{X}_{t^{-}}, \hat{\pi}_{t-}),
$$

$$
\frac{\partial \hat{h}_t}{\partial x} := \frac{\partial h}{\partial x}(t, \hat{X}_t, \hat{\pi}_t), 
\frac{\partial\hat{\mathcal{H}}_t}{\partial x}:=\frac{\partial\hat{\mathcal{H}}}{\partial x} (t, \hat{X}_{t}, \hat{\pi}_{t}, \hat{p}_t, \hat{q}_t, \hat{w}_t),
$$
\medskip
\noindent and we introduce the same notation for their partial derivatives w.r.t. $u$.
We consider the difference $J(\pi) - J(\hat{\pi})$. The idea of the proof is to show that this difference is non-positive for all $\pi \in \mathcal{A}$ so that $\hat{\pi}$ is the optimal control. 

Note that from the definition of the performance function, $J(\pi) - J(\hat{\pi})= A_1 + A_2$ where

\[
\begin{array}{llll}
A_1 := E[\int_0^T \{h_t-\hat{h}_t\} dt ] \\[\medskipamount]
A_2 :=  E[g(X_T) - g(\hat{X}_T)]
\end{array}
\]

\noindent Then, from the definition of the Hamiltonian (see equation \eqref{eqref: Hamiltonian}) and the assumed concavity of $(x, \pi) \mapsto \mathcal{H}(t,x,\pi,\hat{p},\hat{q}, \hat{w})$.

\begin{equation}
\label{eq: I}
\begin{array}{llll}
A_1 &=& E[\int_0^T \{\mathcal{H}_t - \hat{\mathcal{H}}_t - \big(  b_t - \hat{b}_t + \sum_{i=1}^{N_t} Y_i \boldsymbol{1}_{(T_{i-1},T_i]}(t)\lambda_t(\gamma_t-\hat{\gamma}_t) \big)\hat{p}_t \\[\smallskipamount]
&&- (\sigma_t - \hat{\sigma}_t)\hat{q}_t 
- \lambda_t \big( \sum_{i=1}^{N_t} Y_i^2 \boldsymbol{1}_{(T_{i-1},T_i]}(t)(\gamma_t - \hat{\gamma}_t) \\[\medskipamount]
&&+ \sum_{i=1}^{N_t} Y_i \boldsymbol{1}_{(T_{i-1},T_i]}(t)(X_t - \hat{X}_t) \big)\hat{w}_t \} dt] \\[\medskipamount]
&&\leq E[\int_0^T \{ \frac{\partial \hat{\mathcal{H}}_t}{\partial x} (X_t - \hat{X}_t) + \frac{\partial \hat{\mathcal{H}}_t}{\partial \pi}(\pi_t - \hat{\pi}_t) - \big( b_t-\hat{b}_t \\[\medskipamount]
&&+ \sum_{i=1}^{N_t} Y_i \boldsymbol{1}_{(T_{i-1},T_i]}(t)\lambda_t(\gamma_t-\hat{\gamma}_t) \big)\hat{p}_t - (\sigma_t - \hat{\sigma}_t)\hat{q}_t \\[\medskipamount]
&&- \lambda_t \big(   \sum_{i=1}^{N_t} Y_i^2 \boldsymbol{1}_{(T_{i-1},T_i]}(t)(\gamma_t - \hat{\gamma}_t) \\[\medskipamount]
&&+ \sum_{i=1}^{N_t} Y_i \boldsymbol{1}_{(T_{i-1},T_i]}(t)(X_t - \hat{X}_t) \big) \hat{w}_t  \} dt]

\end{array}
\end{equation}

\noindent Furthermore, from the concavity of $g$ and the terminal conditon of the adjoint BSDE \eqref{eq: BSDE},

\begin{equation}
\label{eq: II}
\begin{array}{llll}
     A_2 &=& E[g(X_T)- g(\hat{X}_T)] \\[\smallskipamount]
     &\leq& E[g'(X_T)(X_T - \hat{X}_T)] \\[\smallskipamount]
     &=& E[\hat{p}_T (X_T - \hat{X}_T)]
\end{array}
\end{equation}

\noindent Note that, from integration by parts (see Protter \cite{Protter}, Corollary 2, page 68, for the general semi-martingale case)

\begin{equation}
\label{eq: int_by_parts}
\begin{array}{llll}
    d(\hat{p}_t (X_t - \hat{X}_t)) &=& (X_{t-}-\hat{X}_{t-})d\hat{p}_t + \hat{p}_{t-}d(X_t - \hat{X}_t) + d[\hat{p},X-\hat{X}]_t  \\
\end{array}
\end{equation}
\noindent where $[\hat{p},X-\hat{X}]_t$ is the quadratic covariation of the processes $\hat{p}$ and $X-\hat{X}$ (see Protter \cite{Protter}). Furtermore, from Lemma \ref{lemma: quadratic_variation}:
\begin{equation}
\label{eq: covariation}
d[\hat{p},X-\hat{X}]_t = \hat{q}_t (\sigma_t - \hat{\sigma}_t)dt + \hat{w}_{t-}(\gamma_{t-} - \hat{\gamma}_{t-})d[U]_{t}
\end{equation}

\noindent where $[U]_t$ is the quadratic variation of the process $U$.

\noindent From equation \eqref{eq: II},

\begin{equation}
\label{eq: delta}
\begin{array}{llll}
    A_2 &\leq&  E\big[\hat{p}_T (X_T - \hat{X}_T)] \\[\medskipamount]
     & =& E[\int_0^T (X_{t} - \hat{X}_{t})\big( - \frac{\partial \hat{\mathcal{H}}_t}{\partial x}dt + \hat{q}_t dB_t + \hat{w}_{t-} dU_t \big)
     + \int_0^T \hat{p}_{t} \big( (b_t-\hat{b}_t)dt \\[\medskipamount]
     &&+ (\sigma_t - \hat{\sigma}_t)dB_t + (\gamma_{t-} - \hat{\gamma}_{t-})dU_t \big) \\[\smallskipamount]
     &&+\int_0^T \hat{q}_t (\sigma_t - \hat{\sigma}_t)dt + \int_0^T \hat{w}_{t-}(\gamma_{t-} - \hat{\gamma}_{t-})d[U]_t
     \big]
\end{array}
\end{equation}

\smallskip
\noindent where the equality above follows from the integration-by-parts formula (see equation \eqref{eq: int_by_parts}) as well as the definition of the adjoint BSDE \eqref{eq: BSDE}, the self-exciting state SDE \eqref{eq: SDE_SE} and Lemma \ref{lemma: quadratic_variation}. 

Hence, by combining equations \eqref{eq: I} and \eqref{eq: delta}, and using some algebra as well as the fact that the Brownian motion is a martingale,

\begin{equation}
\label{eq: tilde}
\begin{array}{llll}
    A_1 + A_2  &\leq& E[\int_0^T \big\{ \frac{\partial \hat{\mathcal{H}}_t}{\partial x}(X_t - \hat{X}_t) +  \frac{\partial \hat{\mathcal{H}}_t}{\partial \pi}(\pi_t - \hat{\pi}_t) - \big(b_t-\hat{b}_t + \\[\medskipamount]
    &&\sum_{i=1}^{N_t} Y_i \boldsymbol{1}_{(T_{i-1},T_i]}(t)\lambda_t(\gamma_t - \hat{\gamma}_t)\big)\hat{p}_t -(\sigma_t - \hat{\sigma}_t)\hat{q}_t \\[\medskipamount]
    &&- \lambda_t (\sum_{i=1}^{N_t} Y_i^2 \boldsymbol{1}_{(T_{i-1},T_i]}(t)(\gamma_t - \hat{\gamma}_t) \\[\medskipamount]
    &&+ \sum_{i=1}^{N_t} Y_i \boldsymbol{1}_{(T_{i-1},T_i]}(t)(X_t - \hat{X}_t))\hat{w}_t \big\}dt] \\[\medskipamount]
    && + E[\int_0^T \{ (-\frac{\partial \hat{\mathcal{H}}_t}{\partial x})(X_{t} - \hat{X}_{t}) + \hat{p}_{t}(b_t-\hat{b}_t) + \hat{q}_t (\sigma_t - \hat{\sigma}_t)\}dt \\[\medskipamount]
    &&+ \int_0^T \{(X_{t} - \hat{X}_{t}) \hat{q}_t +  (\sigma_t - \hat{\sigma}_t)\hat{p}_{t}\} dB_t
    \\[\medskipamount]
    &&+ \int_0^T \{ (X_{t-} - \hat{X}_{t-}) \hat{w}_{t-} + (\gamma_{t-} - \hat{\gamma}_{t-})\hat{p}_{t-} \}dU_t \\[\medskipamount]
    &&+ \int_0^T \hat{w}_{t-}(\gamma_{t-} - \hat{\gamma}_{t-})d[U]_t] \\[\medskipamount]
    &=& E[\int_0^T  \frac{\partial \hat{\mathcal{H}}_t}{\partial \pi}(\pi_t - \hat{\pi}_t) dt]
    - E\big[\int_0^T \{ \sum_{i=1}^{N_t} Y_i \boldsymbol{1}_{(T_{i-1},T_i]}(t) \lambda_t(\gamma_t - \hat{\gamma}_t)\hat{p}_t \\[\medskipamount]
    &&+ \lambda_t \big( \sum_{i=1}^{N_t} Y_i^2 \boldsymbol{1}_{(T_{i-1},T_i]}(t)(\gamma_t - \hat{\gamma}_t) \\[\medskipamount]
    &&+ \sum_{i=1}^{N_t} Y_i \boldsymbol{1}_{(T_{i-1},T_i]}(t)(X_t-\hat{X}_t) \big)\hat{w}_t \}dt \\[\medskipamount]
    &&- \int_0^T \{ (X_{t-} - \hat{X}_{t-})\hat{w}_{t-} + (\gamma_{t-} - \hat{\gamma}_{t-})\hat{p}_{t-} \}dU_t \\[\smallskipamount]
    &&- \int_0^T (\gamma_{t-} - \hat{\gamma}_{t-})\hat{w}_{t-} d[U]_t \big]
    
\end{array}
\end{equation}

\noindent From Lemma \ref{lemma: martingales} it follows that

\[
\begin{array}{llll}
    \bar{M}_t &:=& U_t - \sum_{i=1}^{N_t} Y_i \int_0^t \lambda_s \boldsymbol{1}_{(T_{i-1},T_i]}(s) ds \mbox{ and }  \\
    \tilde{M}_t &:=& [U]_t - \sum_{i=1}^{N_t} Y_i^2 \int_0^t \lambda_s \boldsymbol{1}_{(T_{i-1},T_i]}(s)ds 
\end{array}
\]

\noindent are martingales. Hence, 

\begin{equation}
\label{eq: U's}
\begin{array}{llll}
    U_t &=& \bar{M}_t + \sum_{i=1}^{N_t} Y_i \int_0^t \lambda_s \boldsymbol{1}_{(T_{i-1},T_i]}(s) ds \mbox{ and }  \\[\medskipamount]
    [U]_t &=& \tilde{M}_t + \sum_{i=1}^{N_t} Y_i^2 \int_0^t \lambda_s \boldsymbol{1}_{(T_{i-1},T_i]}(s) ds 
\end{array}
\end{equation}

By inserting the expressions in \eqref{eq: U's} into the expression in \eqref{eq: tilde}, and using that $M$ and $\tilde{M}$ are martingales (so their integrals have expectation 0), we get that the expression in \eqref{eq: tilde} becomes

\[
\begin{array}{llll}
E[\int_0^T  \frac{\partial \hat{\mathcal{H}}_t}{\partial \pi}(\pi_t - \hat{\pi}_t) dt] \\[\medskipamount]
- E\big[ \int_0^T \{ \sum_{i=1}^{N_t} Y_i \boldsymbol{1}_{(T_{i-1},T_i]}(t) \lambda_t (\gamma_t - \hat{\gamma}_t)\hat{p}_t + \lambda_t \big( \sum_{i=1}^{N_t} Y_i^2 \boldsymbol{1}_{(T_{i-1},T_i]}(t) (\gamma_t - \hat{\gamma}_t) \\[\medskipamount]
+ \sum_{i=1}^{N_t} Y_i \boldsymbol{1}_{(T_{i-1},T_i]}(t)(X_t - \hat{X}_t) \big)\hat{w}_t \} dt \big] \\[\medskipamount]
+ E \big[ \int_0^T \{ (X_t - \hat{X}_t)\hat{w}_t + (\gamma_t - \hat{\gamma}_t)\hat{p}_t \} \sum_{i=1}^{N_t} Y_i \boldsymbol{1}_{(T_{i-1},T_i]}(t)\lambda_t dt \big]
\\[\medskipamount]
+ E[\int_0^T \{ (X_{t-} - \hat{X}_{t-})\hat{w}_{t-} + (\gamma_{t-} - \hat{\gamma}_{t-})\hat{p}_{t-} \} dM_t] \\[\medskipamount]
+ E\big[ \int_0^T (\gamma_t - \hat{\gamma}_t) \hat{w}_t \sum_{i=1}^{N_t} Y_i^2 \boldsymbol{1}_{(T_{i-1},T_i]}(t) \lambda_t dt \big]
+E[\int_0^T (\gamma_{t-} - \hat{\gamma}_{t-}) \hat{w}_{t-} d\tilde{M}_t]\\[\bigskipamount]
= E[\int_0^T  \frac{\partial \hat{\mathcal{H}}_t}{\partial \pi}(\pi_t - \hat{\pi}_t) dt] \leq 0.
\end{array}
\]
\noindent Where the final inequality follows because $\hat{\pi}$ maximizes $\hat{\mathcal{H}}$ by assumption.

\noindent Since this holds for any $\pi \in \mathcal{A}$, we have found

\[
J(\pi) \leq J(\hat{\pi}) \quad \forall \quad \pi \in \mathcal{A}.
\]

\noindent Hence, $\hat{\pi}$ is an optimal control for our problem.

\end{proof}

\section{An equivalence/necessary maximum principle}
\label{sec: necessary}

In this section, we prove an equivalence principle, sometimes called a necessary maximimum principle, for the self-exciting optimal control problem \eqref{eq: OptControl}. In particular, this is useful because the concavity condition in the sufficient maximum principle, Theorem \ref{thm: suff_max_princ}, may not hold in applications. 

In order to prove the equivalence maximum principle, we need some additional assumptions. In the remaining part of this section, we assume the following:
\begin{assumption}
We make the following additional assumptions for the equivalence maximum principle:
 \begin{enumerate}
    \item For all $\pi \in \mathcal{A}$ and all $\beta \in \mathcal{A}$ that are bounded, there exists a $\delta > 0$ such that
    \[
\pi + y \beta \in \mathcal{A} \quad \forall \quad y \in [0,\delta].
    \]

    \item Also, $\beta_t := \boldsymbol{1}_{[s,T]}(t) \kappa \in \mathcal{A}$, where $\kappa$ is some bounded and measurable random variable. 
    \item  We let  $X_t^{\pi+y\beta}$ and $X_t^{\pi}$ denote the solutions of the self-exciting state SDE \eqref{eq: SDE_SE} corresponding to $\pi+y\beta$ and $\pi$ respectively. Then, we assume that for all $\pi, \beta \in \mathcal{A}$, the following derivative process exists and are in $\mathcal{L}^2([0,T \times \Omega])$:
    \begin{equation}
    \label{eq: derivative_process}
        x_t := \lim_{y \rightarrow 0^+} \frac{X_t^{\pi+y\beta} - X_t^{\pi}}{y} 
    \end{equation}
\end{enumerate}   
\end{assumption}
\medskip

\noindent Again, we introduce some shorthand to ease notation.

\begin{notation}
We use the following shorthand  notation:\newline
$$
\frac{\partial b_t}{\partial x} :=\frac{\partial b}{\partial x}(t, X^{\pi+y\beta}_t, {\pi}_t+y\beta_t), \quad
\frac{\partial \sigma_t}{\partial x} :=\frac{\partial \sigma}{\partial x}(t, X^{\pi+y\beta}_t, {\pi}_t+y\beta_t),
$$
$$
\frac{\partial \gamma_t}{\partial x} :=\frac{\partial \gamma}{\partial x}(t, X^{\pi+y\beta}_{t^{-}}, {\pi}_{t-}+y\beta_t), \quad
\frac{\partial h_t}{\partial x} :=\frac{\partial h}{\partial x}(t, X^{\pi+y\beta}_t, {\pi}_t+y\beta_t),
$$
$$
\frac{\partial \mathcal{H}_t}{\partial x} :=\frac{\partial  \mathcal{H}}{\partial x}(t, X^{\pi+y\beta}_t, {\pi}_t+y\beta_t,p_t,q_t,w_t),
$$
Furthermore, we use corresponding notation as the one above for the partial derivatives w.r.t. $\pi$. For notational convenience, we also define the shorthand

$$
\frac{\partial }{\partial y} X_t^{\pi+y \beta} |_{y=0} := \lim_{y \rightarrow 0^+} \frac{X_t^{\pi+y\beta} - X_t^{\pi}}{y}.
$$
In the remaining part of this section, we will use this notation for various processes depending on the control $\pi$, not just $X_t^{\pi}$.
\end{notation}

From the self-exciting SDE \eqref{eq: SDE_SE}, we derive that the derivative process $x_t$, defined in equation \eqref{eq: derivative_process}, is given by the SDE:
\begin{equation}
\label{eq: derivative_rewritten}
\begin{array}{lll}
dx_t &=& \Big[\frac{\partial b_t}{\partial x}x_t + \frac{\partial b_t}{\partial \pi}\beta_t \Big] dt+ \Big[\frac{\partial \sigma_t}{\partial x} x_t + \frac{\partial \sigma_t}{\partial \pi} \beta_t\Big] dB_t + \Big[\frac{\partial \gamma_{t-}}{\partial x} x_{t^{-} }+ \frac{\partial \gamma_{t-}}{\partial \pi} \beta_{t-}\Big] dU_t \\[\bigskipamount]
x_0 &=& 0.
\end{array}
\end{equation}

\noindent Note that this is a linear self-exciting SDE. By Assumption \ref{assumption: existence_uniqueness} the coefficient functions of the SDE \eqref{eq: derivative_rewritten} are locally Lipschitz. Hence, by Theorem 16.7.6. in Cohen and Elliott \cite{CohenElliott} the SDE \eqref{eq: derivative_rewritten} admits a unique c\`adl\`ag solution. 


As it turns out, it is equivalent for a control to be a critical point for the performance functional $J$ and a critical point of the Hamiltonian. This is the result given in the following equivalence maximum principle:

\begin{theorem}(The equivalence maximum principle)
\label{eq: necessary}
The following two statements are equivalent:

\begin{itemize}
    \item[$(i)$] $$
    \frac{\partial J(\pi + y\beta)}{\partial y} |_{y=0} =0,
    $$
    \item[$(ii)$]
    $$
    \frac{\partial \mathcal{H}_t}{\partial \pi} =0.
    $$
\end{itemize}
\end{theorem}

\begin{proof}
First, we note that from the definition of the performance functional, see equation \eqref{eq: OptControl}, it follows that the expression in item $(i)$ can be written 

$$
\frac{\partial J(\pi + y\beta)}{\partial y} |_{y=0} = \frac{\partial }{\partial y} E[\int_0^T h(t,X_t^{\pi + y\beta}, \pi_t+y\beta_t)dt + g(X_T^{\pi + y\beta})] |_{y=0}.
$$

Define the following,

$$
I_1 := \frac{\partial}{\partial y} E[\int_0^T h(t,X_t^{\pi + y\beta}, \pi_t+y\beta)dt] |_{y=0}
$$
\noindent and
$$
I_2 := \frac{\partial}{\partial y} E[g(X_T^{\pi + y\beta})] |_{y=0}.
$$

By Assumption \ref{assumption: existence_uniqueness}, the coefficient functions of the self-exciting SDE \eqref{eq: SDE_SE} have locally bounded derivatives. Hence, the dominated convergence theorem, in combination with the definition of the derivative process $x_t$ in equation \eqref{eq: derivative_process}, implies that

$$
I_1 = E\big[\int_0^T \{ \frac{\partial h}{\partial x}x_t + \frac{\partial h}{\partial \pi} \beta_t \} dt \big].
$$

\noindent Similarly, the dominated convergence theorem also implies that

$$
I_2 = E[g'(X_T^u)x_T] = E[p_T x_t],
$$

\noindent where the final equality follows from the terminal condition of the BSDE \eqref{eq: BSDE}.

To calculate $E[p_T x_T]$, we use integration by parts (in the semi-martigale framework), see Protter \cite{Protter}, Corollary 2, page 68:
$$
d(p_t x_t) = x_{t-} dp_t + p_{t-} dx_t + d[p,x]_t.
$$

\noindent Recall that, from the adjoint BSDE \eqref{eq: BSDE}

\begin{equation}
\label{eq: mellom_1}
dp_t = - \frac{\partial \mathcal{H}_t}{\partial x} dt + q_t dB_t + w_{t-} dU_t
\end{equation}

\noindent and from the equation \eqref{eq: derivative_rewritten}
\begin{equation}
\label{eq: mellom_2}
dx_t = [\frac{\partial b_t}{\partial x} x_t + \frac{\partial b_t}{\partial \pi} \beta_t]dt + [\frac{\partial \sigma_t}{\partial x}x_t + \frac{\partial \sigma_t}{\partial \pi} \beta_t]dB_t + [\frac{\partial \gamma_{t-}}{\partial x}x_{t-} + \frac{\partial \gamma_{t-}}{\partial \pi} \beta_{t-}]dU_t, \quad x_0=0.
\end{equation}

\noindent Hence, from equations \eqref{eq: mellom_1} and \eqref{eq: mellom_2}, using similar arguments to the proof of Lemma \ref{lemma: quadratic_variation}

$$
\begin{array}{llll}

d[p,x]_t &=& q_t (\frac{\partial \sigma_t}{\partial x}x_t + \frac{\partial \sigma_t}{\partial \pi}\beta_t) dt + w_{t-} (\frac{\partial \gamma_{t-}}{\partial x} x_{t-} + \frac{\partial \gamma_{t-}}{\partial \pi} \beta_{t-}) d[U]_t. \\[\medskipamount]
\end{array}
$$
\smallskip

\noindent Then, from integration-by-parts

$$
\begin{array}{llll}
I_2 &=& E[p_Tx_T] \\[\medskipamount]
&=& E[\int_0^T -\frac{\partial \mathcal{H}_t}{\partial x}x_t dt] + E[\int_0^T x_{t-} w_{t-} dU_t] + E[\int_0^T p_t(\frac{\partial b_t}{\partial x} x_t + \frac{\partial b_t}{\partial \pi} \beta_t)dt] \\[\medskipamount]
&+& E[\int_0^T p_{t-} (\frac{\partial \gamma_{t-}}{\partial x} x_{t-} + \frac{\partial \gamma_{t-}}{\partial \pi} \beta_{t-})dU_t] + E[\int_0^T q_t (\frac{\partial \sigma_t}{\partial x} x_t + \frac{\partial \sigma_t}{\partial \pi} \beta_t) dt] \\[\medskipamount]
&+& E[\int_0^T w_{t-} (\frac{\partial \gamma_{t-}}{\partial x} x_{t-} + \frac{\partial \gamma_{t-}}{\partial \pi} \beta_{t-}) d[U]_t]
\end{array}
$$
\noindent From Lemma \ref{lemma: martingales}, we have $U_t = \bar{M}_t +  \sum_{i=1}^{N_t} Y_i \int_{T_{i-1}}^{T_i} \lambda_s ds$ and
$[U]_t = \tilde{M}_t - \sum_{i=1}^{N_t} Y_i^2 \int_{T_{i-1}}^{T_i} \lambda_s ds$. By inserting this into the expressions above, we find that

\begin{equation}
\label{eq: regning1}
\begin{array}{llll}
I_2 &=& E[\int_0^T -\frac{\partial \mathcal{H}_t}{\partial x}x_t dt] + E[\int_0^T x_t w_t \sum_{i=1}^{N_t} Y_i \boldsymbol{1}_{(T_{i-1},T_i]}(t) \lambda_t dt] \\[\medskipamount]
&+& E[\int_0^T p_t(\frac{\partial b_t}{\partial x} x_t + \frac{\partial b_t}{\partial \pi} \beta_t)dt] \\[\medskipamount]
&+& E[\int_0^T p_t (\frac{\partial \gamma_t}{\partial x} x_t + \frac{\partial \gamma_t}{\partial \pi} \beta_t)\sum_{i=1}^{N_t} Y_i \boldsymbol{1}_{(T_{i-1},T_i]}(t) \lambda_t dt] \\[\medskipamount]
&+& E[\int_0^T q_t (\frac{\partial \sigma_t}{\partial x} x_t + \frac{\partial \sigma_t}{\partial \pi} \beta_t) dt] \\[\medskipamount]
&+& E[\int_0^T w_{t} (\frac{\partial \gamma}{\partial x} x_t + \frac{\partial \gamma}{\partial \pi} \beta_t) \sum_{i=1}^{N_t} Y_i^2 \boldsymbol{1}_{(T_{i-1},T_i]}(t) \lambda_t dt]
\end{array}
\end{equation}

\noindent Also, note that

\begin{equation}
\label{eqref: Hamiltonian2}
\begin{array}{lllll}
    \mathcal{H}_t &=& h_t + \big( b_t + \sum_{i=1}^{N_t} Y_i \boldsymbol{1}_{(T_{i-1},T_i]}(t) \lambda_t \gamma_t \big) p_t 
     + \sigma_t q_t \\[\medskipamount]
     &&+ \lambda_t \big(\sum_{i=1}^{N_t} Y_i^2 \boldsymbol{1}_{(T_{i-1},T_i]}(t) \gamma_t+\sum_{i=1}^{N_t} Y_i \boldsymbol{1}_{(T_{i-1},T_i]}(t)x \big)w_t 
\end{array}
\end{equation}

\noindent Hence, 

\begin{equation}
\label{eq: hamiltonian_derivert}
\begin{array}{lllll}
\frac{\partial \mathcal{H}_t}{\partial x} &=& \frac{\partial h_t}{\partial x} + p_t \big( \frac{\partial b_t}{\partial x} + \frac{\partial \gamma_t}{\partial x} \sum_{i=1}^{N_t} Y_i \boldsymbol{1}_{(T_{i-1},T_i]}(t) \lambda_t \big) + q_t \frac{\partial \sigma_t}{\partial x} \\[\medskipamount]
&&+ w_t \lambda_t \big( \frac{\partial \gamma_t}{\partial x} \sum_{i=1}^{N_t} Y_i^2 \boldsymbol{1}_{(T_{i-1},T_i]}(t) +  \sum_{i=1}^{N_t} Y_i \boldsymbol{1}_{(T_{i-1},T_i]}(t) \big).
\end{array}
\end{equation}

\noindent We return to the calculation in equation \eqref{eq: regning1}, and collect the $x_t$-terms as well as the $\beta_{t-}$-terms. If we compare the resulting expression with the derivative of the Hamiltonian w.r.t. $x$, see equation \eqref{eq: hamiltonian_derivert}, we see that all $x_t$-terms cancel. Therefore, we find that

$$
\begin{array}{llll}
I_1 + I_2 &=& E[\int_0^T \beta_t \big( \frac{\partial h_t}{\partial \pi} + p_t \frac{\partial b_t}{\partial \pi} + p_t \frac{\partial \gamma}{\partial \pi} \sum_{i=1}^{N_t} Y_i \boldsymbol{1}_{(T_{i-1},T_i]}(t) \lambda_t + q_t \frac{\partial \sigma_t}{\partial \pi} \\[\bigskipamount]
&&+ w_t \frac{\partial \gamma_t}{\partial \pi} \sum_{i=1}^{N_t} Y^2_i \boldsymbol{1}_{(T_{i-1},T_i]}(t) \lambda_t \big)dt] \\[\bigskipamount]
&=& E[\int_0^T \beta_t \frac{\partial \mathcal{H}_t}{\partial \pi} dt] \quad \forall \quad \beta \in \mathcal{A}.
\end{array}
$$
\noindent where the last equality follows from the definition of the Hamiltonian, see equation \eqref{eqref: Hamiltonian2} (by differentiating w.r.t. $\pi$).  Hence, the implication from $(ii)$ to $(i)$  follows directly from this expression. The implication from $(i)$ to $(ii)$ follows by letting $\beta_t = \boldsymbol{1}_{[s,T]}(t) \kappa$, where $\kappa$ is bounded and measurable. Then, from $(i)$

$$
0=E[\int_s^T \frac{\partial \mathcal{H}_t}{\partial \pi} \kappa dt].
$$

This holds for all such $\kappa$ and all $s \in [0,T]$. Hence, the equivalence principle follows.

\end{proof}

\section{Application: Log-utility performance functional}\label{sec:Application}

In this section, we consider an application of the previous results. Consider the following linear self-exciting SDE:

\begin{equation}
\label{eq: SDE_SE_application}
    \begin{array}{llll}
        dX_t &=& X_{t-} \big( (\alpha_t - \pi_t)dt + \beta_t dB_t + \kappa_{t-} dU_t \big) \\
         X_0 &=& x_0 > 0.  
    \end{array}
\end{equation}
\noindent where $\alpha, \beta$ and $\kappa$ are bounded and $\mathbb{R}$-valued. In fact, we can find an explicit solution to this SDE for a given control process $\pi_t$. This is the result of the following lemma. 

\begin{lemma}(The solution of the self-exciting SDE \eqref{eq: SDE_SE_application})
\label{lemma: solution}
Consider the self-exciting SDE \eqref{eq: SDE_SE_application} for a given control $\pi_t$. The solution of this SDE is

\[
X_t = X_0 \exp{\big( \int_0^t \{ \alpha_s - \pi_s- \frac{1}{2}\beta_s^2 \} dt  - \frac{1}{2}\int_0^T \kappa_{s-}^2 d[U]_t + \int_0^t \beta_s dB_s + \int_0^t \kappa_{s-} dU_s  \big)}
\]

\end{lemma}

\begin{proof}
By choosing $f(t,x)=\ln(x)$ and using Itô's formula, see Protter \cite{Pr}, we get that

\[
\begin{array}{llll}
df(t,X_t) &=& \frac{1}{X_t}dX_t - \frac{1}{2X_{t-}^2}\big( X_t^2 \beta_t^2 dt + X_{t-}^2 \kappa_{t-}^2 d[U]_t \big) \\[\smallskipamount]
&=& \frac{1}{X_t} X_t\big( (\alpha_t - \pi_t )dt + \beta_t dB_t + \kappa_{t-} dU_t  \big) - \frac{1}{2}\big( \beta_t^2 dt + \kappa_{t-}^2 d[U]_t  \big) 
\end{array}
\]
\noindent where the last equality follows from the definition of the self-exciting SDE \eqref{eq: SDE_SE_application}. Hence,

\[
 d(\ln X_t) = \big( (\alpha_t - \pi_t) - \frac{1}{2}\beta_t^2 \big)dt - \frac{1}{2}\kappa_{t-}^2 d[U]_t + \beta_t dB_t + \kappa_{t-} dU_t
\]
\noindent By integrating both sides, rearranging and taking exponentials, we get 
\[
X_t = X_0 \exp{\big( \int_0^t \{\alpha_s - \pi_s - \frac{1}{2}\beta_s^2 \}dt  - \frac{1}{2}\int_0^T \kappa_{s-}^2 d[U]_t + \int_0^t \beta_s dB_s + \int_0^t \kappa_{s-} dU_s  \big)}
\]

\end{proof}

We can derive an explicit solution to the particular kind of linear self-exciting BSDE coming from the self-exciting stochastic optimal control problem which corresponds to the case when the state equation \eqref{eq: SDE_SE} is linear. 

\begin{lemma}(Solution of the linear self-exciting BSDE)
\label{thm: BSDE_linear}
Consider a self-exciting linear BSDE of the form

\begin{equation}
\label{eq: BSDE_linear}
    \begin{array}{llll}
        dp_t &=& -g(t,p_t, q_t,w_t) dt + q_t dB_t + w_{t-} d U_t \\[\medskipamount]
         p_T &=& F. 
    \end{array}
\end{equation}
\noindent where $F$ is $\mathcal{F}_T$-measurable and where
\[
\begin{array}{llll}
g(t,p_t,q_t,w_t) &=& \varphi_t + (\alpha_t + \sum_{i=1}^{N_t} Y_i \boldsymbol{1}_{(T_{i-1},T_i]}(t)\lambda_t\kappa_t) p_t + \beta_t q_t \\[\medskipamount] 
&&+ w_{t} \lambda_{t} \sum_{i=1}^{N_t}( Y_i + Y_i^2 \kappa_t)\boldsymbol{1}_{(T_{i-1},T_i]}(t)     
\end{array}
\]
Then, the first component of the solution of the BSDE \eqref{eq: BSDE_linear} is given by

\[
p_t = E \big[ \frac{\Gamma_T}{\Gamma_t}F + \int_t^T \frac{\Gamma_s}{\Gamma_t} \varphi_s ds \vert \mathcal{F}_t \big]
\]

\noindent where $\Gamma_t$ is the solution of the linear SDE

\[
    \begin{array}{llll}
        d\Gamma_t &=& \Gamma_{t-} \big( \alpha_t dt + \beta_t dB_t + \kappa_t dU_t \big) \\[\medskipamount]
         \Gamma_0 &=& 1. 
    \end{array}
\]
\end{lemma}

\begin{proof}
The proof is similar to that of {\O}ksendal and Sulem \cite{oksendal2015risk}, Theorem 2.7. By integration by parts (see Protter \cite{Pr}), we find

\[
\begin{array}{llllll}
d\big( \Gamma_t p_t \big) &=& \Gamma_{t-} dp_t+ p_{t-}d\Gamma_t +d \big[ \Gamma, p \big]_t \\[\medskipamount]
&=& \Gamma_{t-}\big( -(\varphi_t + (\alpha_t + \sum_{i=1}^{N_t} Y_i \boldsymbol{1}_{(T_{i-1},T_i]}(t)\lambda_t\kappa_t) p_t + \beta_t q_t \\[\medskipamount] 
&&+ w_{t} \lambda_{t} \sum_{i=1}^{N_t}( Y_i + Y_i^2 \kappa_t)\boldsymbol{1}_{(T_{i-1},T_i]}(t))dt + q_t dB_t + w_{t-} dU_t \big)\\[\medskipamount]
&& + p_{t-}\Gamma_{t-} \big( \alpha_t dt + \beta_t dB_t  + \kappa_t dU_t \big) + \Gamma_{t-} (q_t \beta_t dt + \kappa_t w_{t-} d[U]_t)\\[\medskipamount]
&=& \Gamma_{t-}\big(-\varphi_tdt + (q_t+p_t \beta_t) dB_t \\[\medskipamount] 
&&+ (w_{t-} + \kappa_t p_t) d\bar{M}_t + \kappa_t w_{t-}d\tilde{M}_t\big) 
\end{array}
\]
\noindent where the second equality follows from the SDE for $\Gamma$ and the BSDE for $p$. The third equality follows by cancelling terms and using Lemma \ref{lemma: martingales}. From this, $\Gamma_t p_t + \int_0^t \Gamma_s\varphi_sds$ is a martingale, so

\[
\begin{array}{lllll}
\Gamma_t p_t + \int_0^t \Gamma_s\varphi_sds &=& E \big[ \{ \Gamma_Tp_T + \int_0^T \Gamma_s\varphi_sds \}\vert \mathcal{F}_t   \big] \\[\medskipamount]
&=& E \big[ \{ \Gamma_TF + \int_0^T \Gamma_s\varphi_sds \}\vert \mathcal{F}_t   \big]
\end{array}
\]
\noindent from the BSDE \eqref{eq: BSDE_linear}. This is the result of the lemma.
\end{proof}

We consider a log-utility performance functional and hence would like to solve the following stochastic optimal control problem

\begin{equation}
\label{eq: OptControl_application}
    \max_{\pi \in \mathcal{A}} \quad J(\pi) = E[\int_0^T \ln (X_t \pi_t)dt + \theta \ln (X_T)].
\end{equation}

\noindent The corresponding Hamiltonian is

\[
\begin{array}{llll}
\mathcal{H}_t &=& \ln(x_t \pi_t) + \big( x_t(\alpha_t - \pi_t) + \sum_{i=1}^{N_t} Y_i \boldsymbol{1}_{(T_{i-1}, T_i]}(t) \lambda_t \kappa_t x_t \big)p_t + \beta_t x_t q_t \\
&&+ \lambda_t w_t \big( \sum_{i=1}^{N_t} Y_i^2 \boldsymbol{1}_{(T_{i-1}, T_i]}(t) \kappa_t x_t + \sum_{i=1}^{N_t} Y_i \boldsymbol{1}_{(T_{i-1}, T_i]}(t) x_t \big) 
\end{array}
\]

\noindent Hence, from equation \eqref{eq: BSDE}, the corresponding self-exciting adjoint BSDE is

\begin{equation}
\label{eq: BSDE_application}
    \begin{array}{llll}
        dp_t &=& -\frac{\partial \mathcal{H}}{\partial x} dt + q_t dB_t + w_{t-} d U_t \\[\medskipamount]
         p_T &=& \frac{\theta}{X_T}
    \end{array}
\end{equation}

\noindent where 

\[
\begin{array}{lll}
\frac{\partial \mathcal{H}}{\partial x} &=& \frac{1}{X_t} +(\alpha_t - \pi_t) p_t + \lambda_t w_t \big( \sum_{i=1}^{N_t} Y_i \boldsymbol{1}_{(T_{i-1},T_i]}(t) 
+ \sum_{i=1}^{N_t} Y_i^2 \boldsymbol{1}_{(T_{i-1}, T_i]}(t) \kappa_t \big) \\
&&+  \lambda_t \kappa_t p_t \sum_{i=1}^{N_t} Y_i \boldsymbol{1}_{(T_{i-1}, T_i]}(t) 
+ \beta_t q_t 
\end{array}
\]

\noindent Furthermore, note that

\[
\frac{\partial \mathcal{H}}{\partial \pi} = \frac{1}{\pi_t}
 - p_t X_t
\]
  
\noindent Setting this equal to zero and by using the sufficient maximum principle, Theorem \ref{thm: suff_max_princ}, we get that the optimal control is given by

\begin{equation}
\label{eq: hat_pi}
\hat{\pi}_t = \frac{1}{X_t p_t}
\end{equation}

\noindent Note that from Lemma \ref{lemma: solution}, if we are given a control process $\pi_t$, then we have an explicit expression for $X_t$. By using Lemma \ref{thm: BSDE_linear}, we can derive an expression for the adjoint process $p_t$. However, this expression will depend also on the control process (and on $X_T$, which again is determined by the control). This is the result of the following lemma.

\begin{lemma}(Solution of the adjoint BSDE)
\label{lemma: solution_BSDE}
Consider the adjoint self-exciting BSDE \eqref{eq: BSDE_application}. The solution of the BSDE is 
\[
p_t = \frac{E [ \theta   \vert \mathcal{F}_t ] + T-t}{X_t},
\]
where $X$ is given by \eqref{eq: SDE_SE_application}.
\end{lemma}

\begin{proof}
Let $\Gamma = X$, with the drift $\alpha_t - \pi_t$, $\varphi_t= X_t^{-1}$ and $F=\theta X_T^{-1}$. Then, by Lemma \ref{thm: BSDE_linear} by inserting the framework of the current log-utility application it follows that  
\[
p_t = E \big[ \frac{X_T}{X_t}\frac{\theta}{X_T} + \int_t^T \frac{X_s}{X_t}\frac{1}{X_s} ds\vert \mathcal{F}_t \big].
\]
\end{proof}
It follows that 

\[
\hat{\pi}_t = \frac{1}{E \big[ \theta  \vert \mathcal{F}_t  \big] + T-t}. 
\]
Note that this is the same as in the non-self-exciting case, see e.g. Agram et al. \cite{cherif2020stochastic}. The reason we end up with the same result as in Agram et al. \cite{cherif2020stochastic}, although the state process is more complicated in our framework, is the choice of the log-utility performance functional which simplifies the calculations significantly. 

\bibliographystyle{plain}
\bibliography{SE_control.bib}     

\end{document}